\documentclass[11pt]{article}
\usepackage{amsmath,amssymb,amsfonts,amsthm}

\newtheorem{theorem}{Theorem}
\newtheorem{proposition}[theorem]{Proposition}
\newtheorem{lemma}[theorem]{Lemma}

\begin{document}

\date{}
\title{Polar factorization of conformal and projective maps of the sphere in
the sense of optimal mass transport}
\author{Yamile Godoy and Marcos Salvai~\thanks{%
This work was partially supported by \textsc{Conicet} (PIP
112-2011-01-00670), \textsc{Foncyt} (PICT 2010 cat 1 proyecto 1716) \textsc{%
Secyt Univ.\thinspace Nac.\thinspace C\'{o}rdoba}} \\
{\small {CIEM - FaMAF, Conicet - Universidad Nacional de C\'ordoba }}\\
{\small {Ciudad Universitaria, 5000 C\'{o}rdoba, Argentina}} \\
{\small {ygodoy@famaf.unc.edu.ar, salvai@famaf.unc.edu.ar}}}
\maketitle

\begin{abstract}
Let $M$ be a compact Riemannian manifold and let $\mu ,d$ be the associated
measure and distance on $M$. Robert McCann obtained, generalizing results
for the Euclidean case by Yann Brenier, the polar factorization of Borel
maps $S:M\rightarrow M$ pushing forward $\mu $ to a measure $\nu $: each $S$
factors uniquely a.e.\ into the composition $S=T\circ U$, where $%
U:M\rightarrow M$ is volume preserving and $T:M\rightarrow M$ is the optimal
map transporting $\mu $ to $\nu $ with respect to the cost function $d^{2}/2$%
.

In this article we study the polar factorization of conformal and projective
maps of the sphere $S^{n}$. For conformal maps, which may be identified with
elements of $O_{o}\left( 1,n+1\right) $, we prove that the polar
factorization in the sense of optimal mass transport coincides with the
algebraic polar factorization (Cartan decomposition) of this Lie group. For
the projective case, where the group $GL_{+}\left( n+1\right) $ is involved,
we find necessary and sufficient conditions for these two factorizations to
agree.
\end{abstract}


\smallskip


\section{Introduction}

\noindent \textbf{Optimal mass transport and polar factorization. }Given two
spatial distributions of mass, the problem of Monge and Kantorovich (see for
instance \cite{Villani}) is to transport the mass from one distribution to
the other as efficiently as possible. Here efficiency is measured against a
cost function $c(x,y)$ specifying the transportation tariff per unit mass.
More precisely, let $X$ be a topological space, let $c:X\times X\rightarrow 
\mathbb{R}$ be a nonnegative cost function and let $\mu ,\,\nu $ be finite
Borel measures on $X$ with the same total mass. A map $T:X\rightarrow X$
that minimizes the functional 
\begin{equation*}
T\mapsto \int_{X}c(x,T(x))~d\mu (x)
\end{equation*}%
under the constraint that $T$ pushes forward $\mu $ onto $\nu $ (that is, $%
\nu (B)=\mu (T^{-1}(B))$ for any Borel set in $X$, which is denoted by $%
T_{\#}\mu =\nu $) is called an\textit{\ optimal transportation map between }$%
\mu $ \textit{and }$\nu $. In the following, when it is clear from the
context, we call it simply optimal.

An important particular case is the following: Let $M$ be a compact oriented
Riemannian manifold and let $\mu =$ vol be the Riemannian measure on $M$ and 
$d$ the associated distance and consider the cost $c\left( p,q\right) =\frac{%
1}{2}d\left( p,q\right) ^{2}$.

Let $S:M\rightarrow M$ be a Borel map pushing forward $\mu $ to a measure $%
\nu $. Robert McCann proved in \cite{McCann}, generalizing results for the
Euclidean case by Brenier \cite{Brenier}, that $S$ factors uniquely a.e.\
into the composition $S=T\circ U$, where $U:M\rightarrow M$ is volume
preserving and $T:M\rightarrow M$ is the optimal map transporting $\mu $ to $%
\nu $. This is called the \emph{polar factorization of }$S$\emph{\ in the
sense of optimal mass transport}. For the sake of brevity we call it the
Brenier-McCann polar factorization of $S$.

\bigskip

\noindent \textbf{Polar factorization of conformal maps of the sphere. }An
orientation preserving diffeomorphism $F$ of an oriented Riemannian manifold 
$\left( M,g\right) $ of dimension $n\geq 2$ is said to be \textsl{conformal}
if $F^{\ast }g=fg$ for some positive function $f$ on $M$.

Let $S^{n}$ be the unit sphere centered at the origin of $\mathbb{R}^{n+1}$
and let $S:S^{n}\rightarrow S^{n}$ be a conformal transformation of $S^{n}$.
Let $G=O_{o}\left( 1,n+1\right) $ be the identity component of the group
preserving the symmetric bilinear form of signature $\left( 1,n+1\right) $
on $\mathbb{R}^{n+2}$. The map $S$ can be canonically identified with an
element of $G$, thinking of $S^{n}$ as the projectivization of the light
cone in the Lorentz space $\mathbb{R}^{1,n+1}$. For $A\in G$ and $u\in S^{n}$
(in particular $\left( 1,u\right) $ is a null vector) one defines $A\cdot u$
as the unique $u^{\prime }\in S^{n}$ such that%
\begin{equation}
A\left( 
\begin{array}{c}
1 \\ 
u%
\end{array}%
\right) \in \mathbb{R}\left( 
\begin{array}{c}
1 \\ 
u^{\prime }%
\end{array}%
\right) \text{.}  \label{APuntoU}
\end{equation}%
This is well known; we refer for instance to \cite{Udo} (see also \cite%
{Salvai02}). By definition, the conformal transformations of the circle $%
S^{1}$ are given by the above action of $O_{o}\left( 1,2\right) $ on it.
They coincide with the Moebius maps of the circle, that is, the restrictions
to $S^{1}$ of the Moebius maps of $\mathbb{C}\cup \left\{ \infty \right\} $
preserving the unit disc.

The (algebraic) polar factorization of $G$ is $\exp \left( \mathfrak{p}%
\right) SO\left( n+1\right) $, where $\mathfrak{p}$ is the vector space of
symmetric matrices in the Lie algebra $o\left( 1,n+1\right) $ of $G$, that is%
\begin{equation}
\mathfrak{p}=\left\{ \left( 
\begin{array}{cc}
0 & v^{t} \\ 
v & 0%
\end{array}%
\right) \mid v\in \mathbb{R}^{n}\right\} \text{.}  \label{p}
\end{equation}%
In this context, it is usually called the Cartan decomposition of $G$. We
have the following result:

\begin{theorem}
\label{conformes} Let $S$ be a conformal transformation of the sphere $S^{n}$%
. Then the Brenier-McCann polar factorization of $S$ coincides with the
algebraic polar factorization of $S$.
\end{theorem}

\bigskip

\noindent \textbf{Polar factorization of projective maps of the sphere. }An
orientation preserving diffeomorphism $F$ of an oriented Riemannian manifold 
$M$ of dimension $n\geq 2$ is said to be \textsl{projective} if for any
geodesic $\gamma $ of $M$, $F\circ \gamma $ is a reparametrization (not
necessarily of constant speed) of a geodesic of $M$.

For $n\geq 2$, the projective transformations of $S^{n}$ are exactly those
of the form 
\begin{equation}
p\longmapsto Ap/\left\Vert Ap\right\Vert  \label{Aproyectiva}
\end{equation}%
for some $A\in GL_{+}(n+1)$, the group of linear automorphisms of $\mathbb{R}%
^{n+1}$ with positive determinant. By definition, the projective
transformations of the circle $S^{1}$ are given by the action of $%
GL_{+}\left( 2\right) $ on it as above.

\begin{theorem}
\label{proyectivas}\textbf{\ }Let $S$ be a projective transformation of the
sphere $S^{n}$. Then the Brenier-McCann polar factorization of $S$ coincides
with the algebraic polar factorization $PO$ of $S$ \emph{(}that is, with
positive definite self adjoint $P$ and orthogonal $O$\emph{) }if and only if 
$P$ has at most two distinct eigenvalues.
\end{theorem}

We would like very much to know explicitly, if possible, the Brenier-McCann
polar factorization of the projective map of $S^{2}$ induced by, say, diag~$%
\left( 1,2,3\right) $.

\bigskip

Next we give the main arguments in the proofs of the theorems. A projective
map of $S^{n}$ induced by a positive definite self adjoint transformation of 
$\mathbb{R}^{n+1}$ with at most two distinct eigenvalues preserves meridians
of the sphere through points lying in two fixed orthogonal subspaces. It
turns out to be a particular case of the optimal maps considered in Theorem %
\ref{CorderoProyectiva} in the preliminaries, that we prove using the
characterization by McCann of the optimal maps on Riemannian manifolds
involving $c$-convex potentials. These are a powerful theoretical tool, in
general arduous to deal with in concrete cases, but the fact that the
optimal maps of the circle are well known and the symmetries of our problem
allowed us to apply them.

Conformal maps on $S^{n}$ induced by symmetric transformations of Euclidean
space behave similarly. We prove Theorem \ref{conformes} by verifying that
the conformal map induced by the symmetric part of the polar algebraic
decomposition satisfies the hypotheses of Theorem \ref{CorderoConforme},
which is analogous to Theorem \ref{CorderoProyectiva}.

We comment on the proof of Theorem \ref{proyectivas}. When $P$ has at most
two distinct eigenvalues we check that the projective map induced by $P$
satisfies the hypotheses of Theorem \ref{CorderoProyectiva} using that
conformal maps of the circle double cover the projective maps of the circle.
In the case that $P$ has at least three distinct eigenvalues, we resort to a
necessary condition for a map on a Riemannian manifold $M$ to be optimal:
that its graph be a.e. a Lagrangian submanifold with respect to a certain
symplectic form defined a.e. on $M\times M$ in terms of the cost function.

\bigskip

In \cite{Salvai02}\ and \cite{salvai07} (see also \cite{Emmanuele}) the
second author et al.\negthinspace\ studied force free conformal and
projective motions of $S^{n}$. In particular, they found some geodesics $%
\sigma $ of the groups $SL\left( n+1,\mathbb{R}\right) $ and $O_{o}\left(
1,n+1\right) $ endowed with the not invariant (as it happens with non-rigid
motions) Riemannian metric given by the kinetic metric. The curves $\sigma $
induce curves of measures $t\mapsto \left( \sigma _{t}\right) _{\#}\left( 
\text{vol}\right) $. We wonder whether this is related, with conformal or
projective constraints, with the survey paper \cite{Buttazzo} on
transportation problems in which a given mass dynamically moves from an
initial configuration to a final one (see also \cite{BenamouB} and \cite%
{Brasco}).

\section{Preliminaries}

Let $\left( M,g\right) $, $\mu =$ vol$_{M}$ and $c=d^{2}/2$ be as in the
introduction. We recall the result by McCann in \cite{McCann}, which gives
an expression for the optimal map transporting $\mu $ to a measure $\nu $.
Given a lower semi-continuous function $\phi :M\rightarrow \mathbb{R}$, the
supremal convolution of $\phi $ is the function $\phi ^{c}:M\rightarrow 
\mathbb{R}$ defined by 
\begin{equation}
\phi ^{c}(p)=\sup_{q\in M}\{-c(p,q)-\phi (q)\}\text{.}  \label{c-transf}
\end{equation}%
(We adopt the notation of \cite{Loeper}, using supremal instead of infimal
convolution.)

\begin{theorem}[McCann]
\label{McCann} The optimal map $T$ between $\mu $ and $T_{\#}(\mu )$ can be
expressed as a gradient map, that is, 
\begin{equation*}
T(p)=\text{\emph{Exp}}_{p}\left( \text{\emph{grad}}_{\,p}\phi \right)
\end{equation*}%
a.e., where \emph{Exp} is the geodesic exponential map of $M$ and $\phi $ is
a $c$-convex potential, that is, $\phi :M\rightarrow \mathbb{R}$ is a lower
semi-continuous function satisfying $\phi ^{cc}=\phi $.
\end{theorem}

Cordero-Erausquin characterized in \cite{Cordero} (see also \cite{McCann})
the optimal transportation maps on tori. We write below an equivalent
statement taken
from subsection 6.2 in \cite{N.Q.Le}.

\begin{theorem}[Cordero-Erausquin]
\label{CE} Let $\mu $ and $\nu $ be two probability measures on the torus $%
T^{n}=\mathbb{R}^{n}/2\pi \mathbb{Z}^{n}$ with the same total mass which are
absolutely continuous with respect to the Lebesgue measure and let $%
T:T^{n}\rightarrow T^{n}$ be a map pushing forward $\mu $ to $\nu $. Then $T$
is the optimal transportation map between $\mu $ and $\nu $ if and only if
it lifts a.e. to a map $\widetilde{T}:\mathbb{R}^{n}\rightarrow \mathbb{R}
^{n}$ of the form $\widetilde{T}=$ \emph{grad} $\psi $, where $\psi $ is a
convex function on $\mathbb{R}^{n}$ such that $\psi\left(
x\right) -\left\vert x\right\vert ^{2}/2$ descends to $T^{n}$.
\end{theorem}

Conformal or projective maps of $S^{n}$ induced by self adjoint
transformations of $\mathbb{R}^{n+1}$ (with at most two distinct
eigenvalues, in the projective case) have a particular behavior encompassed
in the following two theorems. Besides, for $n=1$, we have a particular case
of the theorem above. We denote by $\{e_{0},\cdots ,e_{n}\}$ the canonical
basis of $\mathbb{R}^{n+1}$.

\begin{theorem}
\label{CorderoProyectiva}Let $f:\mathbb{R}\rightarrow \mathbb{R}$ be a
smooth $\pi $-periodic odd function with $f^{\prime }+1>0$ and write $%
\mathbb{R}^{n+1}=\mathbb{R}^{k}\oplus \mathbb{R}^{\ell }$. Then $%
T:S^{n}\rightarrow S^{n}$, 
\begin{equation}
T\left( \cos x~u,\sin x~v\right) =\left( \cos \left( x+f\left( x\right)
\right) u,\sin \left( x+f\left( x\right) \right) v\right) \text{,}
\label{Tproyectiva}
\end{equation}%
where $x\in \mathbb{R}$, $u\in \mathbb{R}^{k}$, $v\in \mathbb{R}^{\ell }$
and $\left\Vert u\right\Vert =\left\Vert v\right\Vert =1$, is well defined
and is the optimal transportation map among all maps sending \emph{vol}$%
_{S^{n}}$ to $T_{\#}\left( \emph{vol}_{S^{n}}\right) $.\ 
\end{theorem}

\begin{proof}
First we show that $T$ is well defined. Let $p=\left( \cos x~u,\sin
x~v\right) $ and suppose that $p=\left( \cos y~U,\sin y~V\right) $, where $%
U\in \mathbb{R}^{k}$ and $V\in \mathbb{R}^{\ell }$ are unit vectors. Then 
\begin{equation*}
\cos y=\varepsilon \cos x\text{,\ \ \ \ }U=\varepsilon u\text{,\ \ \ \ }\sin
y=\delta \sin x\text{\ \ \ \ \ and \ \ \ \ \ }V=\delta v
\end{equation*}%
for some $\varepsilon ,\delta =\pm 1$. If $\varepsilon =1$, then $y=\delta
x+2m\pi $ for some $m\in \mathbb{Z}$ and so 
\begin{equation}
T\left( \cos y~U,\sin y~V\right) =\left( \cos \left( \delta x+f\left( \delta
x\right) \right) u,\sin \left( \delta x+f\left( \delta x\right) \right)
\delta v\right) \text{,}  \label{Tz}
\end{equation}%
which coincides with the right hand side of (\ref{Tproyectiva}), as desired 
(we used that $f$ is odd and that a $\pi $-periodic function is in
particular $2\pi $-periodic). Next we consider the case $\varepsilon =-1$.
We have $y=\pi -\delta x+2m\pi $ for some $m\in \mathbb{Z}$ and so $f\left(
y\right) =f\left( \pi -\delta x\right) =f\left( -\delta x\right) =-\delta
f\left( x\right) $, since $f$ is $\pi $-periodic and odd. Finally, using that $\cos \left(\pi -\delta t\right) =-\cos( t)$ and $\sin \left(\pi -\delta t\right) =\delta \sin(t)$, for each $t\in \mathbb{R}$, we obtain (\ref{Tproyectiva}) again.

Now we prove that $T$ is optimal. We consider first the case $k=\ell =1$ ($%
n=1$). We may suppose $u=e_{0}$ and $v=e_{1}$. Then, after the canonical
identification $\mathbb{R}^{2}=\mathbb{C}$ we have that $T\left(
e^{ix}\right) =e^{i\left( x+f\left( x\right) \right) }$, whose lift $%
\widetilde{T}:\mathbb{R}\rightarrow \mathbb{R}$ with $\widetilde{T}\left(
0\right) =0$ (notice that $T\left( 1\right) =1$ since $f$ is odd) is given
by $\widetilde{T}\left( x\right) =x+f\left( x\right) $. Now, a primitive of $%
\widetilde{T}$ is convex since $1+f^{\prime }>0$. Also, $\psi \left( x\right) -x^{2}/2$
descends to the circle, since $\int_{0}^{2\pi }f=0$ ($f$ is odd and $\pi $%
-periodic). By the result of
Cordero-Erausquin in Theorem \ref{CE} for $n=1$, $T$ is optimal.

We return to the general case. We consider the vector field $W$ on $S^n$ defined by%
\begin{equation*}
W\left( p\right) =f\left( x\right) \left( -\sin x~u,\cos x~v\right)
\end{equation*}%
for $p=\left( \cos x~u,\sin x~v\right) $. One can check as above that $W$ is
well defined. We have that 
\begin{equation*}
T\left( p\right) =\text{Exp}_{p}\left( W\left( p\right) \right) \text{.}
\end{equation*}%
By the result of McCann in Theorem \ref{McCann} it suffices to show that $W$
is the gradient of a function $\phi $ with $\phi ^{cc}=\phi $. We propose 
\begin{equation*}
\phi \left( \cos x~u,\sin x~v\right) =\int_{0}^{x}f\left( s\right) ~ds\text{%
, }
\end{equation*}
which is well defined since $\int_{a}^{a+\pi }f\left( s\right) ~ds=0$ for
any $a\in \mathbb{R}$ (indeed, the fact that $f$ is $\pi $-periodic and continuous implies that $\int_{a}^{a+\pi }f\left( s\right) ~ds=\int_{0}^{\pi }f\left( s\right) ~ds$ for any $a\in \mathbb{R}$, and also as $f$ is odd we have that the last integral vanishes). 

We compute 
\begin{eqnarray*}
d\phi _{p}\left( -\sin x~u,\cos x~v\right) &=&\left. \frac{d}{dt}\right\vert
_{0}\phi \left( \left( \cos \left( x+t\right) u,\sin \left( x+t\right)
v\right) \right) \\
&=&\left. \frac{d}{dt}\right\vert _{0}\int_{0}^{x+t}f\left( s\right)
~ds=f\left( x\right) \text{.}
\end{eqnarray*}%
Hence $\left\langle \text{grad}_{\,p}\phi ,\left( -\sin x~u,\cos x~v\right)
\right\rangle =f\left( x\right) $.

Let $X\in T_{p}S^{n}=p^{\bot }$ with $X\bot \left( -\sin x~u,\cos x~v\right) 
$. Suppose that $X=\left( u^{\prime },v^{\prime }\right) $ (in particular, $%
u^{\prime }\bot u$ and $v^{\prime }\bot v$). Let $\alpha $ and $\beta $ be
smooth curves on $S^{k-1}$ and $S^{\ell -1}$ with $\alpha \left( 0\right)
=u,\alpha ^{\prime }\left( 0\right) =u^{\prime }/\cos x,\beta \left(
0\right) =v$ and $\beta ^{\prime }\left( 0\right) =v^{\prime }/\sin x$. We
compute 
\begin{equation*}
d\phi _{p}\left( X\right) =\left. \frac{d}{dt}\right\vert _{0}\phi \left(
\cos x~\alpha \left( t\right) ,\sin x~\beta \left( t\right) \right) =-\left. 
\frac{d}{dt}\right\vert _{0}\int_{0}^{x}f\left( s\right) ~ds=0\text{.}
\end{equation*}%
Therefore grad$~\phi =W$.

Finally, we have to verify that $\phi $ is a $c$-convex potential. For this,
we resort to the case $n=1$ and use $SO\left( n\right) $-invariance. Since
the sphere is compact, given $p\in S^{n}$ there exists $q_{o}\in S^{n}$ such
that 
\begin{equation*}
\phi ^{c}(p)=\sup_{q\in S^{n}}\{-c_{p}(q)-\phi (q)\}=-c_{p}(q_{o})-\phi
(q_{o}),
\end{equation*}%
where $c_{p}(q)=c(p,q)$. Next we observe that $p$ and $q_{o}$ lie in the
same great circle through $\left( \mathbb{R}^{k}\times \left\{ 0\right\}
\right) \cap S^{n}$ and $\left( \left\{ 0\right\} \times \mathbb{R}%
^{k}\right) \cap S^{n}$. In fact, since $q_{o}$ is the maximum of the map 
\begin{equation*}
q\in S^{n}\mapsto -c_{p}(q)-\phi (q)\in \mathbb{R}\text{,}
\end{equation*}%
it follows that $(dc_{p})_{q_{o}}=-(d\phi )_{q_{o}}$. So, their kernels
coincide. Suppose $q_{o}=\left( \xi ,\eta \right) $. We saw above that Ker~$%
(d\phi )_{q_{o}}=\xi ^{\bot }\times \eta ^{\bot }$. 

Also, a straightforward computation yields that Ker~$%
(dc_{p})_{q_{o}}=q_{o}^{\bot }\cap p^{\bot }$. If $q_{o}\neq \pm p$,
then $p\bot q_{o}^{\bot }\cap p^{\bot }=\xi ^{\bot }\times \eta ^{\bot }$,
and so 
\begin{equation*}
p\in \left( \xi ^{\bot }\times \eta ^{\bot }\right) ^{\bot }=\text{span}%
~\left\{ \left( \xi ,0\right) ,\left( 0,\eta \right) \right\} \text{.}
\end{equation*}

We suppose first that $p$ is in the circle $S:=$ span$\left\{
e_{0},e_{k+1}\right\} \cap S^{n}$ and, by the above observation, we see that 
\begin{equation*}
\phi ^{c}(p)=\sup_{q\in S^{n}}\left\{ -c_{p}(q)-\phi (q)\right\} =\max_{q\in
S}~\{-c_{p}(q)-\phi (q)\}=\phi _{o}^{c}(p),
\end{equation*}%
where $\phi _{o}=\left. \phi \right\vert _{S}$. Now, if $p\in S_{p}:=$ span~$%
\left\{ \left( u,0\right) ,\left( 0,v\right) \right\} \cap S^{n}$ with $u,v$
unit vectors, let $R\in SO\left( k\right) \times SO\left( \ell \right) $
such that $R(u,0)=e_{0}$, $R\left( 0,v\right) =e_{k+1}$. Since 
\begin{equation*}
c(p,q)=c(R(p),R(q))\hspace{1cm}\text{and}\hspace{1cm}\phi \circ R(q)=\phi (q)
\end{equation*}%
for all $p,q\in S^{n}$, we have that 
\begin{equation*}
\phi ^{c}(p)=\max_{q\in S_{p}}~\{-c_{p}(q)-\phi (q)\}=\max_{q\in
S}~\{-c_{R(p)}(q)-\phi (q)\}=\phi _{o}^{c}(R(p)),
\end{equation*}%
where the first equality holds again by the above observation. In the same
way, $\phi ^{cc}(p)=\phi _{o}^{cc}\circ R(p)$. We recall that at the
beginning of the proof we saw that $T|_{S}$ is optimal, so $\phi
_{o}^{cc}=\phi _{o}$. Then, for $p\in S^{n}$ we obtain that 
\begin{equation*}
\phi ^{cc}(p)=\phi _{o}^{cc}\circ R(p)=\phi _{o}\circ R(p)=\phi (p)\text{.}
\end{equation*}%
Therefore, $\phi $ is a $c$-convex potential, as we wanted to see, and the
proof of the theorem is complete. 
\end{proof} 

\bigskip

We have a statement similar to the one of the theorem above. The proof is
essentially the same, considering only the case $\varepsilon =1$ and $R\in
SO\left( n\right) $ such that $R\left( e_{0}\right) =e_{0}$ and $R\left(
0,v\right) =e_{1}$.

\begin{theorem}
\label{CorderoConforme}Let $g:\mathbb{R}\rightarrow \mathbb{R}$ be a smooth $%
2\pi $-periodic odd function with $g^{\prime }+1>0$. Then $%
T:S^{n}\rightarrow S^{n}$, 
\begin{equation}
T\left( \cos x,\sin x~v\right) =\left( \cos \left( x+g\left( x\right)
\right) ,\sin \left( x+g\left( x\right) \right) v\right) \text{,}
\label{Tconforme}
\end{equation}%
where $v$ is a unit vector in $\mathbb{R}^{n}$, is well defined and is the
optimal transportation map among all maps sending \emph{vol}$_{S^{n}}$ to $%
T_{\#}\left( \text{\emph{vol}}_{S^{n}}\right) $.
\end{theorem}

\noindent \textbf{Graphs of optimal maps as Lagrangian submanifolds. }Let $M$
be a Riemannian manifold and $c=\frac{1}{2}d^{2}$, as in the introduction.
Kim, McCann and Warren found in \cite{KimWarren} a necessary condition for a
map of $M$ to be optimal, which will be useful in the proof of Theorem \ref%
{proyectivas}. They proved that the graph of the optimal map is calibrated
a.e. by a certain split special Lagrangian calibration on an open dense
subset of $M\times M$ endowed with a certain neutral metric. In particular,
the graph turns out to be Lagrangian with respect to the symplectic form $%
\omega $ on an open dense subset of $M\times M$ defined by $\omega =d\alpha $%
, with $\alpha $ the $1$-form on the same subset of $M\times M$ given by 
\begin{equation*}
\alpha \left( Z\right) =dc\left( X,0\right) =\left( X,0\right) \left(
c\right) \text{,}
\end{equation*}%
where $Z_{\left( p,q\right) }=\left( X_{p},Y_{q}\right) $ after the natural
identification $T_{\left( p,q\right) }\left( M\times M\right) \approx
T_{p}M\times T_{q}M$ (see \cite{Kim}). It is clear that $\omega \left(
\left( u,0\right) ,\left( v,0\right) \right) =\omega \left( \left(
0,u\right) ,\left( 0,v\right) \right) =0$ for any pair of vector fields $u,v$
on $M$.

In the proposition below we describe explicitly $\omega $ for the case where 
$M$ is the sphere $S^{n}$. We will use it in Proposition \ref{Lagrangian}.

\begin{proposition}
\label{omega}Let $p,q$ be a pair of non-antipodal distinct points on $S^{n}$
and let $\gamma :\left[ 0,d\right] \rightarrow S^{n}$ be the unique unit
speed shortest geodesic joining $p$ with $q$. Suppose that $\mathcal{B}%
=\left\{ u_{1},\dots ,u_{n}\right\} $ is an orthonormal basis of $T_{p}S^{n}$
with $u_{1}=\gamma ^{\prime }\left( 0\right) $ and denote by $v_{i}$ the
parallel transport of $u_{i}$ along $\gamma $ from $0$ to $d$ $\mathcal{(}$%
in particular, $v_{1}=\gamma ^{\prime }\left( d\right) $ and $\mathcal{\bar{B%
}}=\left\{ v_{1},\dots ,v_{n}\right\} $ is an orthonormal basis of $%
T_{q}S^{n}\mathcal{)}$. Let $\mathcal{C}$ be the oriented basis of $%
T_{\left( p,q\right) }\left( S^{n}\times S^{n}\right) \approx
T_{p}S^{n}\times T_{q}S^{n}$ obtained by juxtaposing $\mathcal{B}$ with $%
\mathcal{\bar{B}}$. Then%
\begin{equation*}
\left[ \omega _{\left( p,q\right) }\right] _{\mathcal{C}}=\left( 
\begin{array}{cc}
0 & A \\ 
-A & 0%
\end{array}%
\right) \text{,}
\end{equation*}%
where $A=$ $\emph{diag}$~$\left( 1,\frac{d}{\sin d}~I_{n-1}\right) $, with $%
I_{k}$ the $k\times k$ identity matrix.
\end{proposition}

\begin{proof}
Let $U,V$ be vector fields defined on open neighborhoods of $p$ and $q$
in $S^{n}$, respectively. Denote $u=U_{p}$, $v=V_{q}$. If $\sigma $ and $%
\tau $ are curves in $S^{n}$ with $\sigma \left( 0\right) =p$, $\sigma
^{\prime }\left( 0\right) =u$, $\tau \left( 0\right) =q $ and $\tau ^{\prime
}\left( 0\right) =v$, then%
\begin{eqnarray}
\omega _{\left( p,q\right) }\left( \left( u,0\right) ,\left( 0,v\right)
\right) &=&d\alpha _{\left( p,q\right) }\left( \left( u,0\right) ,\left(
0,v\right) \right) =-\left( 0,v\right) _{\left( p,q\right) }\left( dc\left(
U,0\right) \right)  \label{dsdt} \\
&=&-\left. \frac{d}{dt}\right\vert _{0}dc_{\left( p,\tau \left( t\right)
\right) }\left( U_{p},0\right) =-\left. \frac{d}{dt}\right\vert _{0}\left. 
\frac{d}{ds}\right\vert _{0}c\left( \sigma \left( s\right) ,\tau \left(
t\right) \right) \text{,}  \notag
\end{eqnarray}%
where the second equality follows from the fact that $\alpha \left(
0,V\right) =0$ and $\left[ \left( U,0\right) ,\left( 0,V\right) \right] =0$.

Since $S^{n}$ is two-point homogeneous, we may suppose without loss of
generality that $p=e_{0}$, $q=\cos d~e_{0}+\sin d~e_{1}$ with $0<d<\pi $ (so 
$\gamma \left( t\right) =\cos t~e_{0}+\sin t~e_{1}$) and $u_{i}=e_{i}$ ($%
i=1,\dots ,n$). Let $\sigma _{i},\tau _{j}$ be the geodesics in $S^{n}$ with 
$\sigma _{i}\left( 0\right) =p$ and $\sigma _{i}^{\prime }\left( 0\right)
=u_{i}$, $\tau _{j}\left( 0\right) =q$ and $\tau _{j}^{\prime }\left(
0\right) =v_{j}$. We have $\sigma _{1}=\gamma $, $\tau _{1}\left( t\right)
=\gamma \left( t+d\right) $ and 
\begin{equation*}
\sigma _{i}\left( s\right) =\cos s~e_{0}+\sin s~e_{i}\text{,\ \ \ \ \ \ \ \
\ }\tau _{2}\left( t\right) =\cos t\left( \cos d~e_{0}+\sin d~e_{1}\right)
+\sin t~e_{2}\text{.}
\end{equation*}%
We have $A_{ij}=\omega \left( \left( u_{i},0\right) ,\left( 0,v_{j}\right)
\right) $. By the $SO\left( n-1\right) $-symmetries of $S^{n}$ fixing the
trajectory of $\gamma $ it suffices to show that $A_{11}=1$, $%
A_{12}=A_{21}=A_{32}=0$ and $A_{22}=\frac{d}{\sin d}$.

For any manifold $M$, if $\gamma $ is length minimizing on an open interval
containing $\left[ 0,d\right] $ we have $d\left( \gamma \left( s\right)
,\gamma \left( d+t\right) \right) =d+t-s$ for $s,t$ near $0$ and so by (\ref%
{dsdt}) 
\begin{equation*}
A_{11}=-\left. \frac{d^{2}}{dsdt}\right\vert _{\left( 0,0\right) }\frac{1}{2}%
\left( d+t-s\right) ^{2}=1\text{.}
\end{equation*}%
On the other hand, since $c\left( \cdot ,\cdot \right) =\frac{1}{2}\arccos
^{2}\left\langle \cdot ,\cdot \right\rangle $, we have by (\ref{dsdt}) that $%
A_{ij}=f_{ij}^{\prime }\left( 0\right) $, where%
\begin{equation*}
f_{ij}\left( t\right) =-\left. \frac{d}{ds}\right\vert _{0}c\left( \sigma
_{i}\left( s\right) ,\tau _{j}\left( t\right) \right) =\frac{\arccos \left(
\left\langle \sigma _{i}\left( 0\right) ,\tau _{j}\left( t\right)
\right\rangle \right) }{\sqrt{1-\left\langle \sigma _{i}\left( 0\right)
,\tau _{j}\left( t\right) \right\rangle ^{2}}}\left\langle \sigma
_{i}^{\prime }\left( 0\right) ,\tau _{j}\left( t\right) \right\rangle \text{.%
}
\end{equation*}%
Since $\left\langle \sigma _{i}^{\prime }\left( 0\right) ,\tau _{j}\left(
t\right) \right\rangle $ vanishes for $\left( i,j\right) =\left( 2,1\right)
,\left( 3,2\right) $, $f_{21}=f_{32}\equiv 0$. Also, $f_{12}$ is an even
function and so $f_{12}^{\prime }\left( 0\right) =0$. Hence, $%
A_{21}=A_{32}=A_{12}\equiv 0$. Now, putting $\cos d\cos t=\cos \left(
x\left( t\right) \right) $ for a function $x$ with values in the interval $%
\left( 0,\pi \right) $ (in particular $x\left( 0\right) =d$ and $x^{\prime
}\left( 0\right) =0$), we have

\begin{equation*}
A_{22}=f_{22}^{\prime }\left( 0\right) =\left. \frac{d}{dt}\right\vert _{0}%
\frac{\arccos \left( \cos x\left( t\right) \right) }{\sqrt{1-\cos
^{2}x\left( t\right) }}\sin t=\left. \frac{d}{dt}\right\vert _{0}\frac{%
x\left( t\right) }{\sin x\left( t\right) }\sin t=\frac{d}{\sin d}\text{.}
\end{equation*}%
Therefore the coefficients of the matrix $A$ are as desired.
\end{proof}

\section{Proofs of the theorems}

\begin{proof}[Proof of Theorem \ref{conformes}] 
Let $S\in G$. Thus, $S=\exp (A)O$, where $A\in \mathfrak{p}$ (see (\ref{p}))  and $O\in
SO(n+1)$. Clearly, $O$ preserves $\text{vol}_{S^{n}}$. So, by the uniqueness
of the Brenier-McCann polar factorization (see the introduction), we have to
prove that $\exp (A)$ is the optimal transportation map between $\text{vol}%
_{S^{n}}$ and $S_{\#}(\text{vol}_{S^{n}})$. Without lost of generality we
can take $A=a\left( 
\begin{array}{cc}
0 & e_{0}^{t} \\ 
e_{0} & 0%
\end{array}%
\right) \in \mathfrak{p}$, where $a>0$. In fact, if $R\in SO\left( n\right) $
satisfies $R\left( ae_{0}\right) =v$, then $A$ and $\left( 
\begin{array}{cc}
0 & v^{t} \\ 
v & 0%
\end{array}%
\right) $ are conjugate by the isometry diag~$\left( 1,R\right) $ of $S^{n}$%
, and isometries preserve optimality.

Next we compute $\exp (A)$ as a map of the sphere. Putting 
\begin{equation*}
H_{t}=\left( 
\begin{array}{cc}
\cosh at & \sinh at \\ 
\sinh at & \cosh at%
\end{array}%
\right)
\end{equation*}%
we have that $\exp \left( tA\right) =$ diag~$\left( H_{t},I_{n}\right) $.

Let $q=\left( u_{0},\dots ,u_{n}\right) \in S^{n}$ and let $\gamma \left(
t\right) =\exp \left( tA\right) \left( q\right) $. By (\ref{APuntoU}), 
\begin{equation*}
\gamma \left( t\right) =\left( \sinh at+u_{0}\cosh at,u_{1},\dots
,u_{n}\right) /h\left( u_{0},t\right)
\end{equation*}%
with $h(u,t)=u\sinh at+\cosh at$. Now, $\gamma $ is the integral curve
through $q$ of the vector field $V$ on the sphere defined by 
\begin{equation*}
V\left( p\right) =\left. \frac{d}{dt}\right\vert _{0}\exp \left( tA\right)
\left( p\right) =ae_{0}-\left\langle ae_{0},p\right\rangle p\text{.}
\end{equation*}%
Hence, 
\begin{equation*}
\left\Vert \gamma ^{\prime }\left( t\right) \right\Vert =\left\Vert V\left(
\gamma \left( t\right) \right) \right\Vert ={\frac{a\sqrt{1-u_{0}^{2}}}{%
h(u_{0},t)}}\text{.}
\end{equation*}

Since $\gamma $ lies on a meridian through $e_{0}$, which is a geodesic, and 
$V$ vanishes at $\pm e_{0}$, the distance $D\left( q\right) $ from $q$ to $%
\exp \left( A\right) \left( q\right) $ is the length of $\left. \gamma
\right\vert _{\left[ 0,1\right] }$. Hence, $\exp (A)(q)=\text{Exp}_{q}(U(q))$%
, where $U\left( q\right) $ is a tangent vector pointing in the direction of 
$\gamma ^{\prime }\left( 0\right) $ with $\Vert U(q)\Vert =D\left( q\right) $%
. We have 
\begin{equation*}
D\left( \cos x,\sin x~v\right) =\int_{0}^{1}{\frac{a\left\vert \sin
x\right\vert }{h(\cos x,t)}}\,dt.
\end{equation*}%
The appropriate choice of sign yields that $\exp \left( A\right) $ equals $T$
as in (\ref{Tconforme}) with 
\begin{equation}
g\left( x\right) =-\int_{0}^{1}{\frac{a\sin x}{h(\cos x,t)}}\,dt\text{.}
\label{g}
\end{equation}

Now we apply Theorem \ref{CorderoConforme} to prove that $\exp \left(
A\right) $ is optimal. Since $g$ is odd and $2\pi $-periodic, we have to
verify only that $g^{\prime }+1>0$. We compute 
\begin{equation*}
g^{\prime }(x)=-\int_{0}^{1}\frac{d}{dx}{\frac{a\sin x}{h(\cos x,t)}}%
\,dt=-\int_{0}^{1}\frac{h_{t}(\cos x,t)}{h(\cos x,t)^{2}}~dt=\dfrac{1}{\cos
x\sinh a+\cosh a}-1\text{,}
\end{equation*}%
where $h_{t}=\partial h/\partial t$. Hence $g^{\prime }+1>0$, as desired.
\end{proof}

\bigskip

\begin{proof}[Proof of Theorem \ref{proyectivas}] The polar decomposition of $S\in
GL_{+}\left( n+1\right) $ is $S=PO$, where $O\in SO\left( n+1\right) $ and $P
$ is a positive definite self adjoint linear transformation. As in the
conformal case, we may suppose that $P$ is diagonal and we have to prove
that the induced operator $T\left( q\right) =P\left( q\right) /\left\Vert
P\left( q\right) \right\Vert $ on $S^{n}$ is optimal.

We consider first the case when $P$ has at least three distinct eigenvalues.
As a consequence of the necessary condition for optimality stated in the
preliminaries we have that $T$ is not optimal, since otherwise the graph of $%
T$ would be Lagrangian a.e., contradicting Proposition \ref{Lagrangian}
below.

So now we assume that $P$ has exactly two distinct eigenvalues, say $\lambda 
$ and $\mu $, with respective eigenspaces of dimensions $k$ and $\ell $ (the
case when $P$ is a multiple of the identity is trivial). We may suppose that 
\begin{equation*}
P=\text{diag}~\left( e^{\frac{a}{2}}I_{k},e^{-\frac{a}{2}}I_{\ell }\right)
=\exp B\text{,}
\end{equation*}%
where $B=\frac{a}{2}$ diag~$\left( I_{k},-I_{\ell }\right) $ for some $a\in 
\mathbb{R}$. In fact, $T$ does not change if we take instead of $P$ a
positive multiple $cP$ of it (we chose $c=1/\sqrt{\lambda \mu }$ and $a=\log
\left( \lambda /\mu \right) $). We compute 
\begin{equation}
T\left( \cos x~u,\sin x~v\right) =\frac{\left( e^{a/2}\cos x~u,e^{-a/2}\sin
x~v\right) }{\left( e^{a}\cos ^{2}x+e^{-a}\sin ^{2}x\right) ^{1/2}}\text{,}
\label{expAmedio}
\end{equation}%
where $u\in S^{k-1}$ and $v\in S^{\ell -1}$. We see that $T$ preserves the
meridian $S^{n}\cap $ span~$\left\{ u,v\right\} $ and moreover it has the
form (\ref{Tproyectiva}) for some $\pi $-periodic odd function $f$. Now,
identifying span~$\left\{ u,v\right\} $ with $\mathbb{C}$ ($u=1,v=i$) we
have by Lemma \ref{DoubleC} below that 
\begin{equation*}
e^{i2\left( x+g\left( x\right) \right) }=e^{i\left( 2x+f\left( 2x\right)
\right) }
\end{equation*}%
where $g$ is the $2\pi $-periodic odd function in (\ref{g}). Since $f\left(
0\right) =g\left( 0\right) =0$, we have that $f\left( x\right) =2g\left( 
\frac{1}{2}x\right) $ for all $x$. Now, $g^{\prime }+1>0$ by the proof of
Theorem \ref{conformes}, hence $f^{\prime }+1>0$. Therefore, $T$ is optimal
by Theorem \ref{CorderoProyectiva}. 
\end{proof}

\bigskip

Next we state the lemma we used in the proof above. It is well known that
there is a double covering morphism $PSL\left( 2,\mathbb{R}\right)
\rightarrow O_{o}\left( 1,2\right) $. We only need the morphism restricted
to some subgroups isomorphic to the circle. Let 
\begin{equation*}
A=\frac{a}{2}\left( 
\begin{array}{cc}
1 & 0 \\ 
0 & -1%
\end{array}%
\right) \text{\ \ \ \ \ \ and\ \ \ \ \ \ }B=a\left( 
\begin{array}{cc}
0 & 1 \\ 
1 & 0%
\end{array}%
\right)
\end{equation*}%
with $a>0$. Then $\exp A$ and $\exp B$ induce a projective and a conformal
map of $S^{1}$ as in (\ref{Aproyectiva}) and (\ref{APuntoU}), respectively.
The following lemma asserts that the first one double covers the second one.

\begin{lemma}
\label{DoubleC}Let $A$ and $B$ be as above and let $\rho :S^{1}\rightarrow
S^{1}$, $\rho \left( x,y\right) =\left( x^{2}-y^{2},2xy\right) $ \emph{(}%
that is, $\rho \left( z\right) =z^{2}$ after the identification $\mathbb{R}%
^{2}=\mathbb{C}$\emph{)}. Then the following diagram commutes%
\begin{equation*}
\begin{array}{ccc}
S^{1} & \overset{\exp A}{\longrightarrow } & S^{1} \\ 
\downarrow \rho &  & \downarrow \rho \\ 
S^{1} & \overset{\exp B}{\longrightarrow } & S^{1}\text{,}%
\end{array}%
\end{equation*}%
where we are considering the induced conformal and projective maps of the
circle.
\end{lemma}

\begin{proof} The statement is well-known; we sketch the proof for the sake of
completeness. We have that the projective map on $S^{1}$ induced by $\exp
\left( A\right) $ applied to $\left( \cos x,\sin x\right) $ equals the right
hand side of (\ref{expAmedio}) with $u=v=e_{1}=1$. We also have 
\begin{equation*}
\exp \left( B\right) \cdot \left( \cos x,\sin x\right) =\frac{\left( \cosh
a\cos x+\sinh a,\sin x\right) }{\sinh a\cos x+\cosh a}\text{.}
\end{equation*}%
Now, a straightforward computation yields the commutativity of the diagram.
\end{proof}

\bigskip

We used the following proposition in the proof of Theorem \ref{proyectivas}.
It involves the symplectic form $\omega $ considered in the preliminaries.

\begin{proposition}
\label{Lagrangian}Let $P$ be a positive definite self adjoint operator of $%
\mathbb{R}^{n+1}$ with at least three distinct eigenvalues and let $T$ be
the projective map induced by $P$ on $S^{n}$. Then there exists an open
dense subset $W$ of $S^{n}$ such that \emph{graph}\,$\left( dT_{p}\right) $
is \emph{not} a Lagrangian subspace of $T_{p}S^{n}\times T_{T\left( p\right)
}S^{n}$ with respect to $\omega _{\left( p,T\left( p\right) \right) }$ for
any $p\in W$. In particular \emph{graph}$\,\left( T\right) $ is not a
Lagrangian submanifold a.e. of $S^{n}\times S^{n}$.
\end{proposition}

\begin{proof} 
We may suppose without lost of generality that $Pe_{i}=\lambda
_{i}e_{i}$ for $i=0,\dots ,n$ and that $\lambda _{0},\,\lambda
_{1},\,\lambda _{2}$ are positive and pairwise different. Let $W$ be the
open subset of $S^{n}$ consisting of all points whose coordinates are all
different from zero. Let $p\in W$ and let $q=T\left( p\right) \in W$ (which
is different from $p$ and $-p$). Let $\mathcal{B}$ and $\mathcal{\bar{B}}$
be as in Proposition \ref{omega}. Let $\mathcal{B}^{\prime }$ be the ordered
basis consisting of all the elements of the basis $\mathcal{\bar{B}}$,
except that $v_{1}$ is substituted for $V_{1}=\frac{\sin d}{d}v_{1}$, and
let $\mathcal{C}^{\prime }$ be the juxtaposition of $\mathcal{B}$ and $%
\mathcal{B}^{\prime }$. Then the matrix of $\omega _{\left( p,q\right) }$
with respect to $\mathcal{C}^{\prime }$ is a multiple of $\left( 
\begin{array}{cc}
0 & I_{n} \\ 
-I_{n} & 0%
\end{array}%
\right) $. It is well-known that the graph of a linear transformation $%
L:T_{p}S^{n}\rightarrow T_{q}S^{n}$ is Lagrangian for $\omega _{\left(
p,q\right) }$ if and only if the matrix of $L$ with respect to the bases $%
\mathcal{B}$ and $\mathcal{B}^{\prime }$ is symmetric.

A straightforward computation yields 
\begin{equation}
dT_{p}(v)=\frac{1}{\left\Vert P(p)\right\Vert }\left( P(v)-\left\langle
P(v),q\right\rangle q\right) =\frac{1}{\left\Vert P(p)\right\Vert }\text{pr}%
_{q^{\bot }}(P(v))  \label{dT}
\end{equation}%
for any $p\in S^{n}$ and $v\in T_{p}S^{n}$, where pr$_{q^{\bot }}$ is the
orthogonal projection onto $q^{\bot }$. The vectors $u_{1}\in T_{p}S^{n}$
and $v_{1}\in T_{q}S^{n}$ as in Proposition \ref{omega} are 
\begin{equation*}
u_{1}=\frac{q-\left\langle p,q\right\rangle p}{\Vert q-\left\langle
p,q\right\rangle p\Vert }\ \text{\ \ \ \ \ \ \ \ and\ \ \ \ \ \ \ \ \ \ }%
v_{1}=\frac{\left\langle p,q\right\rangle q-p}{\Vert \left\langle
p,q\right\rangle q-p\Vert }\text{.}
\end{equation*}%
Now we write $p=\left( x,y\right) $ with $x\in \mathbb{R}^{3}$ and $P\left(
x,y\right) =\left( P_{1}(x),P_{2}(y)\right) $, where $P_{1}=$ diag~$\left(
\lambda _{0},\lambda _{1},\lambda _{2}\right) $. It is easy to verify that
we can take vectors $u_{2}\in T_{p}S^{n}$ and $v_{2}\in T_{q}S^{n}$ as in
Proposition \ref{omega} as follows:%
\begin{equation*}
u_{2}=v_{2}=\left( P_{1}(x)\times x,0\right) /\left\Vert P_{1}(x)\times
x\right\Vert \text{.}
\end{equation*}%
Now we verify that the matrix of $dT_{p}$ with respect to the bases $%
\mathcal{B}$ and $\mathcal{B}^{\prime }$ (recall that $V_{1}=\frac{\sin d}{d}%
v_{1}$) is not symmetric. We call 
\begin{equation*}
a=\left\Vert P(p)\right\Vert \text{,}\hspace{1cm}b=\Vert q-\left\langle
p,q\right\rangle p\Vert =\Vert \left\langle p,q\right\rangle q-p\Vert \ \ \
\ \ \ \ \text{and\ \ \ \ \ \ \ }c=\left\Vert P_{1}(x)\times x\right\Vert \text{%
.}
\end{equation*}%
Straightforward computations using (\ref{dT}) yield that 
\begin{eqnarray*}
\left\langle dT_{p}(u_{1}),v_{2}\right\rangle  &=&\frac{1}{abc}\left\langle
P_{1}^{2}\left( x\right) ,P_{1}(x)\times x\right\rangle  \\
\left\langle dT_{p}(u_{2}),V_{1}\right\rangle  &=&\frac{\sin d}{d}\frac{1}{%
abc}\left\langle q,p\right\rangle \left\langle P_{1}^{2}\left( x\right)
,P_{1}(x)\times x\right\rangle \text{.}
\end{eqnarray*}%
We compute $\left\langle P_{1}^{2}\left( x\right) ,P_{1}(x)\times
x\right\rangle =\left( \lambda _{0}-\lambda _{1}\right) \left( \lambda
_{0}-\lambda _{2}\right) \left( \lambda _{1}-\lambda _{2}\right)
x_{0}x_{1}x_{2}\neq 0$. Hence $\left\langle dT_{p}(u_{1}),v_{2}\right\rangle
=\left\langle dT_{p}(u_{2}),V_{1}\right\rangle $ if and only if $%
\left\langle q,p\right\rangle \sin d=d$, or equivalently 
\begin{equation*}
\sin \left( 2d\right) =2\cos d\sin d=2d\text{,}
\end{equation*}%
which holds only for $d=0$. Therefore the matrix of $dT_{p}$ with respect to
the bases $\mathcal{B}$ and $\mathcal{B}^{\prime }$ is not symmetric for all 
$p\in W$, as desired. 
\end{proof}

\end{document}